\documentclass[11pt]{article}
\usepackage{amsfonts,amsmath,tikz}
\usepackage{epsfig}
\usepackage{latexsym}
\usepackage{color}
\usepackage{epsf,amssymb}

\usepackage{amsthm}
\usepackage{yfonts}
\usepackage{hyperref}
\usepackage{ifthen}

\newtheorem{theorem}{Theorem}
\newtheorem{lemma}[theorem]{Lemma}
\newtheorem{observation}[theorem]{Observation}
\newtheorem{conjecture}[theorem]{Conjecture}

\newtheorem{cor}[theorem]{Corollary}
\newtheorem{proposition}[theorem]{Proposition}
\newtheorem{question}{Question}

\def\vertex(#1){\put(#1){\circle*{2}}}
\def\vertexo(#1){\put(#1){\circle{2}}}
\def\vert(#1){\put(#1){\circle*{1.5}}}
\def\verto(#1){\put(#1){\circle{1.5}}}
\def\lab(#1)#2{\put(#1){\makebox(0,0)[c]{#2}}}
\newcommand{\cart}{\, \Box \,}

\newcommand{\gt}{\gamma_t}

\newcommand{\rk}{\gamma_{{\rm r}k}}

\newcommand{\cp}{\,\square\,}

\newcommand{\krt}{\gamma_{k{\rm rt}}}
\newcommand{\trt}{\gamma_{2{\rm rt}}}

\newcommand{\Sgraph}{\mathcal{T}}

\tikzset{My Style/.style={draw, circle, fill=black,scale=0.3}} 

\title{Bounding the $k$-rainbow total domination number}

\author{
Kerry Ojakian\\
\small \it Department of Mathematics and Computer Science\\ 
\small \it Bronx Community College (CUNY)\\ 
\small \it Bronx, NY, U.S.A\\
\small \tt kerry.ojakian@bcc.cuny.edu \\
\and
Riste \v Skrekovski \\
\small \it FMF, University of Ljubljana, 1000 Ljubljana, Slovenia \\
\small \it Faculty of Information Studies, 8000 Novo Mesto, Slovenia \\
\small \tt skrekovski@gmail.com  \\
\and
 Aleksandra Tepeh \\
\small \it University of Maribor, FERI, 2000 Maribor, Slovenia \\
\small \it Faculty of Information Studies, 8000 Novo Mesto, Slovenia \\
\small \tt aleksandra.tepeh@um.si \\
}

\begin{document}
\maketitle

\begin{abstract}
Recently the notion of $k$-rainbow total domination was introduced for a graph $G$, motivated by a desire to reduce the problem of computing the total domination number of the generalized prism $G \cart K_k$ to an integer labeling problem on $G$.  In this paper we further demonstrate usefulness of the labeling approach, presenting bounds on the rainbow total domination number in terms of the total domination number, the rainbow domination number and the rainbow total domination number, as well as the usual domination number, where the latter presents a generalization of a result by Goddard and Henning (2018). We establish Vizing-like results for rainbow domination and rainbow total domination. By stating a Vizing-like conjecture for rainbow total domination we present a different viewpoint on Vizing's original conjecture in the case of bipartite graphs.
\end{abstract}


\section{Introduction and preliminaries}

\newcommand{\new}[1]{{\textcolor{red}{#1}}}

Domination is a topic in graph theory with extensive research activity. Already in 1998,
a classic book \cite{fund-1998} by Haynes et al. surveyed over 1200 papers. It is not surprising that this number is increasing rapidly as domination presents one of the most applicable branches of graph theory. For example, graphs can be used to model locations which 
exchange some resource along its edges.  In such an application,
ordinary domination is an optimization
problem to determine the minimum number of locations necessary in order for each location to contain the resource
or be adjacent to a location containing the resource.
However, there are situations where certain additional requirements must be fulfilled, thus
numerous variations of domination, motivated by real life as well as theoretical applications,
developed over the time. In this paper we study the recently introduced concept of $k$-rainbow total domination, compare it with known domination parameters, and show that it leads to an interesting viewpoint on the famous Vizing conjecture.

We begin with some general definitions and notation for graphs, followed by various definitions of domination in graphs.
Graphs considered in this paper are finite, simple and undirected.  For a graph $G$, we let $V(G)$ denote its set of vertices and $E(G)$ denote its set of edges. 
For a graph $G$, let $N_G(v)$ be the open neighborhood of vertex $v$ in graph $G$, that is the vertices in $G$ which are adjacent to $v$.  When $G$ is apparent from context, we may just write $N(v)$. The closed neighborhood $N_G[v]$ is $N_G(v)\cup\{v\}$.
For graphs $G$ and $H$, the Cartesian Product $G\cart H$ is the graph with vertex set $V (G)\times V (H)$. Vertices $(g, h)$ and $(g',h')$ are adjacent in $G\cart H$ if and only if either $g = g'$ and $hh'\in E(H)$ or $h = h'$ and $gg'\in E(G)$.
For an integer $k$, we let $[k]$ refer to the set $\{1,2,\ldots,k\}$, and sometimes refer to the elements of $[k]$ as colors. 
We denote the power set of $[k]$ by $2^{[k]}$.

A dominating set of a graph $G$ is a subset $D$ of $V(G)$ such that every vertex not in $D$ is adjacent to some vertex in $D$. The \textit{domination number}, $\gamma(G)$, is the minimum cardinality of a dominating set of $G$. If every vertex of $G$ is adjacent to a vertex in $D$, then $D$ is called a \textit{total dominating set} of $G$, and the minimum cardinality of a total dominating set of $G$ is the \textit{total domination number}, $\gt(G)$, \cite{HenSur09}
 
One of the well studied domination invariants is $k$-rainbow domination, introduced by Bre\v{s}ar, Henning and Rall~\cite{bhr-2008}.
For a positive integer $k$, a \textit{$k$-rainbow dominating function} (or $k$RDF, for short) of a graph $G$ is a function 
$f : V(G) \rightarrow 2^{[k]}$, such that for any vertex $v$ with $f(v)=\emptyset$ we have $\bigcup_{u\in N(v)} f(u) = [k]$. Let $||f||=\sum_{v\in V(G)}|f(v)|$; we refer to $||f||$ as the \emph{weight} of $f$.
The \textit{$k$-rainbow domination number}, $\rk(G)$, of a graph $G$ is the minimum value of $||f||$ over all $k$-rainbow dominating functions of $G$.   A $k$RDF of weight $\rk(G)$ is called a $\rk$-function.
For example, $\gamma_{r2}(C_4) = 2$ since an optimal choice is to assign $\{1\}$ to one vertex $v$ and assign $\{ 2 \}$ to the vertex not adjacent to $v$ in $C_4$.  It was observed in \cite{bhr-2008}  that for all $k\geq 1$, 
$$\gamma_{rk}(G)=\gamma (G \cart K_k).$$
Since the seminal paper \cite{bhr-2008} introduced this invariant, it has been studied extensively, for example: 
its algorithmic properties \cite{chang10,greedy}, relevant graph operations and families \cite{btks-2007,TKSRT-lex}, and general properties \cite{philip,filo,wx-2010}.

A \textit{$k$-rainbow total dominating function} $f$ of a graph $G$ (a $k$RTDF for short) was introduced in \cite{San19} as a $k$-rainbow dominating function satisfying an additional condition that for every $v\in V(G)$ such that $f(v)=\{i\}$ for some $i\in [k]$, there exists some
$u\in N(v)$ such that $i\in f(u)$. The weight of a $k$RTDF is as in the case of a $k$RDF,  $||f||=\sum_{v\in V(G)}|f(v)|$.
The minimum weight of a $k$RTDF is called the {\em $k$-rainbow  total domination number of $G$}, $\krt(G)$. A $k$RTDF of weight $\krt(G)$ is called a \textit{$\krt$-function}. 
For example, $\gamma_{2rt}(C_4) = 4$ since an optimal choice is to pick two vertices and assign each $\{1,2\}$; another optimal choice simply assigns $\{ 1 \}$ to every vertex.
The definition of the $k$-rainbow total domination number is motivated by wanting to understand total domination in the generalized prism; in
\cite{San19} it was observed that for all $k \ge 1$,
\begin{equation} \label{cpt}
\krt(G)=\gt(G \cart K_k).
\end{equation}

The main point of this paper is to compare the $k$-rainbow total domination number to other notions of domination, in particular to usual domination (i.e~$\gamma$), to total domination  (i.e.~$\gamma_t$), and to $k$-rainbow domination  (i.e.~$\gamma_{rk}$).  In Section~\ref{sec_itself} we calculate some domination numbers of the complete bipartite graph $K_{a,b}$ and a variant of $K_{a,b}$.  The reason for these calculations is mostly for showing the tightness of bounds in subsequent sections.  In Section~\ref{sec_k_total} we bound the $k$-rainbow total domination number in terms of the other kinds of domination numbers.  In Section~\ref{sec_dom} we obtain a lower bound on the $k$-rainbow total domination number in terms of the usual domination number, a generalization of the Goddard and Henning~\cite{GodHen18} result which applied to the case of $k = 2$.  We also investigate the tightness of this lower bound and consider what happens when the graph is bipartite.
In Section~\ref{sec_vizing}  we consider Vizing's 1968 conjecture that for any graphs $G$ and $H$,
        \begin{equation}
				\gamma(G) \, \cdot \, \gamma(H) \le  \gamma(G\cp H).
				\end{equation}
Variations on the Vizing conjecture have been considered by other authors, for example, Ho~\cite{Ho08} in the case of total domination, and 
Bre\v{s}ar et al.~\cite{viz} and Pilipczuk et al.~\cite{philip} in the case of $k$-rainbow domination.
We investigate what happens in the case of $k$-rainbow total domination, proving a simple Vizing-like result, and discussing a stronger Vizing-like conjecture, which in the case of bipartite graphs coincides with the original one.

\section{Special graph classes}
\label{sec_itself}

In this section we compute some domination invariants for $K_{a,b}$ and a related graph we call $K^+_{a,b}$.
Some of the results of this section are interesting in their own right, however, the results of this section will be used in later sections, mostly to demonstrate that certain bounds are tight.
In many calculations we use an alternative way to measure the weight of a $k$-rainbow dominating function $f$: 
$$||f|| = \sum_{S \subseteq [k]}| S|\,|V_S|\,$$
where $V_S=\{x \in V(G)\,:\,f(x)=S\}$.
If a vertex $v$ is assigned a set $S$ by the function $f$, we may refer to $S$ as the \emph{label} of $v$, and when clear from context we may abbreviate labels by removing the set braces, e.g. we would typically abbreviate the label $\{1,2\}$ by just $12$. A vertex $v$ with $f(v)=\emptyset$ will be referred to as \emph{empty vertex\emph}.

For integers $a$ and $b$ such that $a,b \ge 1$, let $K_{a,b}$ be the complete bipartite graph where one of the partite sets, say $A$, consists of vertices $x_1,x_2,\ldots, x_a$, and $B$ is the other partite set. 
We define the graph $K^+_{a,b}$ to be $K_{a,b}$ with 
new vertices $y_1,y_2,\ldots, y_a$ and new edges $x_iy_i$ for every $i\in [a]$; see Figure~\ref{Kab+}.

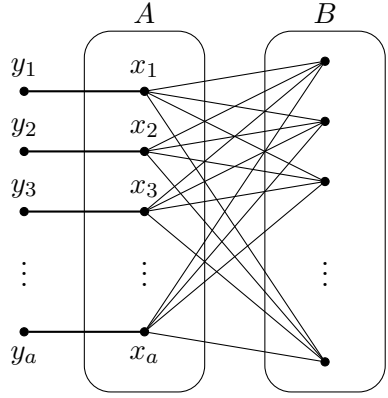
\begin{figure}[ht!]
\begin{center}
\begin{tikzpicture}[scale=0.8]

\draw[black,rounded corners=10]
     (0,0) rectangle (2,6);
		\draw (1,6.3) node {$A$};

\draw[black,rounded corners=10]
     (3,0) rectangle (5,6);
		\draw (4,6.3) node {$B$};

\draw (-1,2.1) node {$\vdots$};
\draw (1,2.1) node {$\vdots$};
\draw (4,2.1) node {$\vdots$};

\node [My Style, name=a, label=below:$x_a$ ]   at (1,1) {};		
\node [My Style, name=b, label=above:$x_3$]   at (1,3) {};	
\node [My Style, name=c, label=above:$x_2$]    at (1,4) {};	
\node [My Style, name=d, label=above:$x_1$]   at (1,5) {};	

\node [My Style, name=a1, label=below:$y_a$ ]   at (-1,1) {};		
\node [My Style, name=b1, label=above:$y_3$]   at (-1,3) {};	
\node [My Style, name=c1, label=above:$y_2$]    at (-1,4) {};	
\node [My Style, name=d1, label=above:$y_1$]   at (-1,5) {};

\draw[thick] (a) -- (a1);
\draw[thick] (b) -- (b1);
\draw[thick] (c) -- (c1);
\draw[thick] (d) -- (d1);

\node [My Style, name=a2]   at (4,0.5) {};		
\node [My Style, name=b2]   at (4,3.5) {};	
\node [My Style, name=c2]   at (4,4.5) {};	
\node [My Style, name=d2]   at (4,5.5) {};

\draw[thin] (a) -- (a2);
\draw[thin] (a) -- (b2);
\draw[thin] (a) -- (c2);
\draw[thin] (a) -- (d2);

\draw[thin] (b) -- (a2);
\draw[thin] (b) -- (b2);
\draw[thin] (b) -- (c2);
\draw[thin] (b) -- (d2);

\draw[thin] (c) -- (a2);
\draw[thin] (c) -- (b2);
\draw[thin] (c) -- (c2);
\draw[thin] (c) -- (d2);

\draw[thin] (d) -- (a2);
\draw[thin] (d) -- (b2);
\draw[thin] (d) -- (c2);
\draw[thin] (d) -- (d2);
		
\end{tikzpicture}
\end{center}
\caption{The graph $K^+_{a,b}$.}
\label{Kab+}
\end{figure}

\begin{lemma} \label{thm_Kplus} $\gamma(K^+_{a,b}) = a$  and   $\gamma_{k\rm{rt}}(K^+_{a,b}) = 2a$ for $a,b\ge 1$ and $2 \le k \le a$.
\end{lemma}

\begin{proof}
The set $A$ is clearly a dominating set of $K^+_{a,b}$, thus $\gamma(K^+_{a,b}) \leq a$. On the other hand, every dominating set of $K^+_{a,b}$ must contain either $y_i$ or $x_i$ for every $i\in [a]$, thus $\gamma(K^+_{a,b}) \geq a$, which proves the first equality. 

Now let $f:V(K^+_{a,b})\rightarrow 2^{[k]}$ be a function defined by $f(v)=\emptyset$ for every $v\in B$, and for every $i \in [a]$, $f(x_i)=f(y_i)=\{j\}$ for 
some $j\in [k]$, so that for every $j\in [k]$ there exists $i$ with $f(x_i)=\{j\}$. Clearly $f$ is a $k$RTDF, thus $\gamma_{k\rm{rt}}(K^+_{a,b})\leq 2a$. Finally, note that $\gamma_{k\rm{rt}}(K^+_{a,b})$ cannot be less than $2a$, since each pair of vertices $x_i,y_i$ contributes at least $2$ to the weight of any $k$RTDF of $K^+_{a,b}$.
\end{proof}

\begin{lemma} \label{lem-Kab-2k}
Suppose $a, b,$ and $k$ are positive integers, $k\geq 1$ and $a\leq b$.  If $a + b \le k$, then 
$\rk(K_{a,b}) = a + b$.  If $a + b > k$, then we have

$$\rk(K_{a,b}) = 
\begin{cases}
2k, & \hbox{if $a \ge 2k$} \\
\hbox{$\max\{a,k\}$,} & \hbox{if $a < 2k$.} \\
\end{cases}
$$ 

\end{lemma}
\begin{proof}

Let $\{A,B\}$ be the bipartition of $V(K_{a, b})$ with $A=\{x_1,x_2, \ldots, x_a\}$ and $B=\{y_1,y_2, \ldots, y_b\}$, where all the edges are between $A$ and $B$.
We consider the following cases.

\begin{itemize}

\item
If $a + b \le k$, the assertion holds by a simple property for a general graph $G$, that if $|V(G)|\leq k$, then $\rk(G)=|V(G)|$.

\item
Let $a + b > k$, $a \ge 2k$ and let $f:V(K_{a, b})\rightarrow 2^{[k]}$ be defined with $f(x_1)=f(y_1)=[k]$ and $f(v)=\emptyset$ for every $v\in V(K_{a, b})\setminus \{x_1,y_1\}$. Since $f$ is clearly a $k$RDF we have $\rk(K_{a,b})\leq 2k$. 
Now let $g$ be an arbitrary $\rk$-function of $K_{a, b}$. If both $A$ and $B$ contain an empty vertex, or neither of them contains such a vertex, we clearly have $||g||\geq 2k$. 
If exactly one of $A$ or $B$ contains an empty vertex, then $||g|| \ge a \ge 2k$, so $\rk(K_{a,b})\geq 2k$.

\item
Let $a + b > k$ and $a < 2k$. Then obviously $\rk(K_{a, b}) \ge k$. If $a\leq k$,  $f:V(K_{a, b})\rightarrow 2^{[k]}$ defined with $f(x_i)=\{i\}$ for $i\in [a-1]$, $f(x_a)=\{a,a+1,\ldots,k\}$ and $f(v)=\emptyset$ for every $v\in B$, is a $k$RDF. Therefore in this case we infer $\rk(K_{a,b})=k$. If $a>k$, then $f:V(K_{a, b})\rightarrow 2^{[k]}$ defined with $f(x_i)=\{i\}$ for every $i\in [k]$, $f(x_i)=\{1\}$ for every $i\in \{k+1,k+2,\ldots,a\}$, and $f(v)=\emptyset$ for every $v\in B$, is a $k$RDF, thus $\rk(K_{a, b}) \leq a$.
Assume for contradiction that $\rk(K_{a, b}) < a$, and let $g$ be an arbitrary $\rk$-function. Then there 
is an empty vertex in $A$ and in $B$
since $b\geq a>k$. Thus $||g||\geq 2k$, contradicting $||g|| < a < 2k$. Therefore $\rk(K_{a,b})=a$ if $a>k$.
\end{itemize}
\end{proof}

\begin{proposition}  \label{thm3}
Suppose $a, b,$ and $k$ are positive integers, $k\geq 2$ and $a\leq b$.  If $a + b \le k$, then 
$\krt(K_{a,b}) = a + b$. If $a + b > k$, then we have
$$\krt(K_{a,b}) = 
\begin{cases}
k, & \hbox{if $a \le \bigl \lfloor \frac{k}{2} \bigr \rfloor$} \\
a+\bigl \lceil \frac{k+1}{2}\bigr \rceil,  & \hbox{if $\bigl \lfloor \frac{k}{2} \bigr \rfloor < a <  \bigl \lceil \frac{3k-2}{2}   \rceil$} \\
2k, & \hbox{if $a \ge \bigl \lceil \frac{3k-2}{2}   \bigr \rceil$.} \\
\end{cases}
$$ 
\end{proposition}

\begin{proof} If $a + b \le k$ the assertion is immediate (see Observation 4 in \cite{San19}). Now let $a + b > k$ and let $\{A,B\}$ be the bipartition of $V(K_{a, b})$ with $A=\{x_1,x_2, \ldots, x_a\}$ and $B=\{y_1,y_2, \ldots, y_b\}$, where $a\leq b$, and  all the edges are between $A$ and $B$.

First we consider the case when $a \le \bigl \lfloor \frac{k}{2} \bigr \rfloor$. Since $a + b > k$, we clearly have $\krt(K_{a, b})\geq k$. Let $f:V(K_{a, b})\rightarrow 2^{[k]}$ be defined with $f(x_i)=\{2i-1,2i\}$ for every $i\in [a-1]$, $f(x_a)=\{2a-1,2a,\ldots, k\}$, and $f(y_i)=\emptyset$ for every $i\in [b]$. Since $f$ is a $k$RTDF and $||f||=k$ we also have $\krt(K_{a,b})\leq k$. Thus $\krt(K_{a,b})= k$ if $a \le \bigl \lfloor \frac{k}{2} \bigr \rfloor$.

\emph{For the remainder of the proof we assume that 
$a \ge \bigl \lfloor \frac{k}{2} \bigr \rfloor + 1 = \bigl \lceil \frac{k+1}{2} \bigr \rceil$.} 
Before we consider the remaining two cases for the value of $a$,
we prove the following three claims.

\medskip

\textit{Claim 1. \ $\krt(K_{a,b})\leq 2k$}.

\medskip

Let $f:V(K_{a, b})\rightarrow 2^{[k]}$ be defined with 
$f(x_1)=f(y_1)=[k]$ and $f(v)=\emptyset$ for every $v \in V(K_{a,b})\setminus \{x_1,y_1\}$. It is straightforward to see that $f$ is a  $k$RTDF, proving Claim 1.

\medskip

\textit{Claim 2. \ $\krt(K_{a,b})\leq a+ \bigl \lceil \frac{k+1}{2}  \bigr \rceil$}.

\medskip 

Let $f:V(K_{a, b})\rightarrow 2^{[k]}$ be defined with 
 $f(x_i)=\{2i-1,2i\}$ if $i\in \{1,2,\ldots, \bigl \lfloor \frac{k}{2} \bigr \rfloor\}$, 
$f(x_i)=\{k\}$ if $i\in \{ \bigl \lfloor \frac{k}{2} \bigr \rfloor +1, \ldots, a\}$,  $f(y_1)=\{k\}$ and $f(y_i)=\emptyset$ for $i \in \{2,3,\ldots,b\}$. 
It is easy to see that $f$ is a $k$RTDF. Thus 
$$
\krt(K_{a, b}) \le ||f|| = 2\cdot \Bigl\lfloor \frac{k}{2}  \Bigr\rfloor+ 1 \cdot \left(a-\Bigl\lfloor \frac{k}{2}  \Bigr\rfloor \right)+1=
a + \Bigl \lceil \frac{k+1}{2} \Bigr \rceil, 
$$
proving Claim 2.

\medskip

\textit{Claim 3. \ If $f$ is a $k$RTDF on $K_{a,b}$ such that all of $A$ consists of non-empty vertices or all of $B$ consists of non-empty vertices, then $||f|| \ge a+ \bigl \lceil \frac{k+1}{2}  \bigr \rceil$. }

\medskip

Let $f$ be a $\krt$-function on $K_{a,b}$ such that $|f(x_i)|\geq 1$ for  every $i\in [a]$; the argument is similar if all the vertices in $B$ are non-empty.  
Suppose that in $A$ there are exactly $s$ vertices which are singleton sets (i.e. suppose $|f(x_1)| = \cdots = |f(x_s)| = 1$ and $|f(x_i)| \ge 2$ for $s < i \le a$). Without loss of generality, assume that the colors appearing among the singletons are $1,2, \ldots, r$, for some $r \le k$.  If $s = 0$, then $||f|| \ge 2a \ge a+ \bigl \lceil \frac{k+1}{2}  \bigr \rceil$, so
we are done.  Thus we assume that $s \ge 1$; also recall that $f$ is a minimum weight $k$RTDF.  
In $B$ we must have exactly one occurrence of each of the  colors $1,2, \ldots, r$, and no other colors.  Among the labels in $A$ which are not singletons we can assume that collectively they include  each color from $\{r+1, \ldots, k \}$ exactly once, and contain no other colors.  To see why we can make that assumption, consider a vertex in $A$ with a non-singleton label $L$, such that $x, y \in L$ where $x$ and $y$ are distinct and $y$ appears on some other vertex in $A$.  We can remove $y$ from $L$ and add the color $x$ to any label in $B$ to arrive at a new $\krt$-function.
So we will assume that in $A$, $f$ consists of the $s$ singletons, and then the rest of $A$ contains each color from $\{r+1, \ldots, k \}$ exactly once.
Thus $||f|| = s + k$, since we have the $s$ singletons in $A$, the colors $1,2, \ldots,r$ appearing exactly once in $B$, and the colors $r+1, \ldots, k$ appearing exactly once in $A$.
We have at least one singleton, so $r \ge 1$, and thus the non-singletons in
$A$ use at most $\bigl \lfloor \frac{k-1}{2} \bigr \rfloor$ vertices of $A$, thus
$s \ge a - \bigl \lfloor \frac{k-1}{2} \bigr \rfloor$.
We can compute as follows
$$||f|| = s + k \ge a - \Bigl \lfloor \frac{k-1}{2} \Bigr \rfloor + k 
= a + \Bigl \lceil \frac{k+1}{2} \Bigr \rceil, 
$$
completing the proof of Claim 3.

\medskip

We turn back to the proof of the theorem, considering the last two cases for the value of $a$.  
By Claims 1 and 2, we have the upper bounds on $\krt(K_{a, b})$ in both cases.
For the lower bounds, we first observe the following lower bounds on a  $k$RTDF $g$
depending on how it labels $A$ and $B$ (the first three lower bounds are simple observations, and 
the fourth is just a statement of  Claim 3).

\begin{enumerate}

\item If $g$ has empty vertices in both $A$ and $B$ then $||g|| \ge 2k$.

\item If $g$ has empty vertices in all of $A$ or all of $B$, then $||g || \ge 2a$.

\item If $g$ has no empty vertices then $||g|| \ge 2a$.

\item If $g$ has no empty vertices in $A$ or no empty vertices in $B$ then $||g|| \ge a+\bigl \lceil \frac{k+1}{2} \bigr \rceil$.

\end{enumerate}
Notice that the four conditions above are not mutually exclusive, but do cover all possible labellings of $A$ and $B$;
thus to show a lower bound, it suffices to cover the above cases.

If $\bigl \lfloor \frac{k}{2} \bigr \rfloor < a < \bigl \lceil \frac{3k-2}{2}  \bigr \rceil$, then
we refer to the above 4 conditions and note that 
$2a = a + a \ge a +\bigl \lceil \frac{k+1}{2} \bigr \rceil$, thus covering Conditions 2 and 3.  Condition 4 already matches this case.
To deal with Condition 1, we note its impossibility for a $\krt$-function $g$, because (using Claim 2) 
$|| g || \le a + \bigl \lceil \frac{k+1}{2} \bigr \rceil < 
\bigl \lceil \frac{3k-2}{2}  \bigr \rceil + \bigl \lceil \frac{k+1}{2} \bigr \rceil = 2k$.

If  $a \geq  \bigl \lceil \frac{3k-2}{2}   \bigr \rceil$, then 
we refer to the above 4 conditions and note that the following calculations suffice in order to cover all the conditions
(recall that $k \ge 2$):
$$2a \ge  2 \cdot \Bigl \lceil \frac{3k-2}{2}   \Bigr \rceil \geq 2k \ \ \hbox{ and } \ \
a + \Bigl \lceil \frac{k+1}{2} \Bigr \rceil \geq \Bigl \lceil \frac{3k-2}{2}   \Bigr \rceil + \Bigl \lceil \frac{k+1}{2} \Bigr \rceil
=2k.$$

\end{proof}

\section{Bounds relating to  $k$-rainbow total domination}
\label{sec_k_total}

In this section we study bounds on the $k$-rainbow total domination number in terms of other domination numbers. 
Shao et al.~\cite{shao14}  proved that 
$\gamma_{{\rm r}k'}(G) \leq \gamma_{{\rm r}k}(G)+(k'-k)\lfloor \frac{\gamma_{{\rm r}k}(G)}{k} \rfloor$ and consequently $\gamma_{{\rm r}k'}(G) \leq \frac{k'}{k}\gamma_{{\rm r}k}(G)$, for
a connected graph $G$ and positive integers $k$ and $k'$ such that $k'\geq k$ (see Theorem 1 and Corollary 1 in \cite{shao14}). 
Using the same proof we can prove the same relations hold for $k$-rainbow total domination.  For completeness, we briefly
recall their proof.
Let $f$ be a $\krt$-function of $G$ and $a_i$ the number of vertices $u$ for which $i\in f(u)$. We may without loss of generality assume that
$a_1\geq a_2\geq  \cdots \geq a_k$. Since $\krt(G)=a_1+a_2+  \cdots + a_k$ we conclude that $ka_k\leq \krt(G)$. Now, by adding the colors $k+1,k+1,\ldots,k'$ to the label of any vertex $u$ such that $k \in f(u)$,
we clearly get a $k'$RTDF whose weight is $\krt(G)+(k'-k) a_k$, therefore 
$\gamma_{k'{\rm rt}}(G) \leq  \krt(G)+(k'-k) \frac{\krt(G)}{k}=\frac{k'}{k}\krt(G)$.
Summarizing this result and adding an observation we conclude the following.

\begin{proposition} \label{prop_frac_bounds}
Let $k$ and $k'$ be integers such that $1 \le k \le k'$. Then the following bounds hold:
$$\gamma_{k'rt}(G) \le \dfrac{k'}{k} \gamma_{krt}(G)  \quad \textit{ and } \quad \gamma_{rk'}(G) \le \dfrac{k'}{k} \gamma_{rk}(G).$$ 
Furthermore, the bounds are tight when  $G$ is $K_{a,b}$ with $a,b \ge 2k'$.
\end{proposition}

\begin{proof}

The upper bounds follow from the Shao et al. proof sketched above. 
To prove the tightness of the bounds
 we provide a family of graphs for which the equalities are attained. 
In both cases, we take the graphs $G$ to be $K_{a,b}$, where $a,b \ge 2k'$. The tightness follows from Lemma~\ref{lem-Kab-2k} and Proposition~\ref{thm3}, which imply  $\gamma_{rk}(G) = \gamma_{krt}(G) = 2k$ and 
 $\gamma_{rk'}(G) = \gamma_{k'rt}(G) = 2k'$.

\end{proof}

Bringing the above results together with another observation, we can give a full description of how the 
rainbow domination numbers compare when the parameter is increased by 1.

\begin{theorem} \label{thm_k_to_kminus}
For any graph $G$, and integer $k$ such that $k \ge 2$, we have the following: 
\begin{enumerate}
\item
$\gamma_{(k-1) {\rm rt}}(G) \le \krt(G)\le \frac{k}{k-1}\gamma_{(k-1) {\rm rt}}(G)$,
\item
$\gamma_{{\rm r}(k-1)}(G) \le \rk(G) \le \frac{k}{k-1}\gamma_{{\rm r}(k-1)}(G)$.
\end{enumerate}
Moreover, both upper bounds are tight for $K_{a,b}$ where $a,b\geq 2k$, and the first lower bound is tight for
$k \ge 3$, when the graph is $K^+_{k,b}$ and $b \ge 1$.
\end{theorem}

\begin{proof} The upper bounds and their tightness follow from Proposition~\ref{prop_frac_bounds}.  Now we discuss the lower bounds.

We prove that $\krt(G) \geq \gamma_{(k-1) {\rm rt}}(G)$. Let $f$ be a $\krt$-function on $G$. We use $f$ to define $g:V(G)\rightarrow2^{[k-1]}$ by letting $g(v) = f(v)$ as long as $k \not \in f(v)$; if $k \in f(v)$ then
$$g(v) = 
\begin{cases}
(f(v) \setminus \{ k \}) \cup \{ 1 \}, & \hbox{if $k-1 \in f(v)$} \\
(f(v) \setminus \{ k \}) \cup \{ k-1 \}, & \hbox{otherwise.}
\end{cases}
$$ 
If $g(v)=\emptyset$, clearly all the colors from $[k-1]$ appear in the neighborhood of $v$. If $g(v)=\{i\}$ and $i \in [k-2]$, then by the definition of $g$ there exists a neighbor $u$ of $v$ such that $i\in g(u)$. If $g(v)=\{k-1\}$, then we have one of the following situations: either $f(v)=\{k-1\}$, so there exists $u\in N_G(v)$ such that $k-1\in f(u)$ and therefore also $k-1\in g(u)$, or $f(v)=\{k\}$ so there exists $u\in N_G(v)$ such that $k\in f(u)$, and thus $k-1\in g(u)$ again. Thus $g$ is a $(k-1)$RTDF and we have $\gamma_{(k-1) {\rm rt}}(G) \le ||g||\leq ||f|| =\krt(G)$, which shows the first upper bound.
To prove the second upper bound, that $\gamma_{{\rm r}(k-1)}(G) \le \rk(G)$, is even simpler, since by taking a $\rk$-function  $f$ on $G$ we can turn it to a $(k-1)$RDF simply by replacing every occurrence of the element $k$ by $k-1$.
  
Now we discuss the tightness of the lower bounds. For the first inequality, when $k \ge 3$, we obtain
$\gamma_{(k-1) {\rm rt}}(G) = \krt(G)$ when $G$ is $K^+_{k,b}$ (recall Lemma~\ref{thm_Kplus}).
When $k = 2$ the equality states that $\gamma_t(G)=\trt(G)$; there exist graphs satisfying the equality, which were in fact characterized in  Theorem~$3$ of \cite{LuHou}.
For the second inequality, when $k\geq 3$ we take integers $a, b$ such that $k < a < 2k \le b$, so by Lemma~\ref{lem-Kab-2k}
$\gamma_{{\rm r}(k-1)}(K_{a,b}) = \rk(K_{a,b})=a$.  When $k = 2$,
 we are interested in graphs $G$ that satisfy $\gamma(G)=\gamma_{r2}(G)$. Such graphs exist and were characterized by Hartnell and Rall, see Theorem~$4$ in~\cite{hart04}.

\end{proof}

The following simple corollary will lead to an interesting question.

\begin{cor} \label{thm_rainbow_to_total}
For any graph $G$ without isolated vertices,
$\gamma_t(G) \le \gamma_{krt}(G) \le k \gamma_t(G)$, and the upper bound is tight.
\end{cor}
\begin{proof}
Recall that if $G$ does not contain isolated vertices, then $\gamma_{1rt}(G)=\gamma_t(G)$. Therefore Theorem~\ref{thm_k_to_kminus} implies the lower bound. It is easy to observe the upper bound by assigning the entire set $[k]$ to each vertex in a total dominating set of $G$.
The upper bound is tight when 
$G$ is $K_{a,b}$ ($a, b \ge 2k$), since $\gamma_t(G) = 2$ is immediate, while  
$\gamma_{krt}(G) = 2k$ follows from Proposition~\ref{thm3}.
\end{proof}

\noindent
Notice that in the last corollary there is no claim about the tightness of the lower bound. As we have already mentioned in the proof of Theorem~\ref{thm_k_to_kminus}, the lower bound in the case when $k=2$ can be attained and all such graphs were characterized in \cite{LuHou}.
For $k \ge 3$, the lower bound of Corollary~\ref{thm_rainbow_to_total} does not appear to be tight, which leads to the following question (where we conjecture that  $b(k) > 1$).

\begin{question}
Find a function $b(k)$ such that for $k \ge 3$, the following bound is true and tight for connected graphs $G$:
$$b(k) \cdot \gamma_t(G) \le \gamma_{krt}(G).$$
\end{question}
\noindent
Note that $\gamma_t(C_n) \geq \frac{n}{2}$ and $\gamma_{krt}(C_n) \le n$, so $b(k) \le 2$.

In the next proposition and question we consider how the $k$-rainbow domination number compares to the $k$-rainbow total domination number.

\begin{proposition} 
\label{rk-krt}
For any graph $G$ and any integer $k$ such that $k \ge 2$ the following holds:
$$\rk(G) \le \krt(G) \le 2\rk(G).$$
Furthermore, the lower bound is tight for every $k$, and the upper bound is tight only for $k = 2$, that is,
for $k \ge 3$, $\krt(G) < 2\rk(G)$.
\end{proposition}

\begin{proof}
As we have already explained, the lower bound follows by definitions of both invariants.
To see that the upper bound holds, let $f$ be a $\rk$-function on $G$. Then a $k$-rainbow total dominating function $g$ can be constructed from $f$ by adding an element to every singleton label of $f$, so $g$ has no singleton labels.
Then $\krt(G)\leq ||g||\leq 2||f||=2\rk(G)$.

The tightness of the lower bound holds when $G$ is $K_{a,b}$ for $b \ge a \ge 2k$, since then we have
$\rk(G) = \krt(G) = 2k$, using Lemma~\ref{lem-Kab-2k} and Proposition~\ref{thm3}.
The tightness of the upper bound, when $k=2$, follows by taking $G$ to be $K_{2,b}$, where $b \ge 2$, since by  Proposition~\ref{thm3} and Lemma~\ref{lem-Kab-2k} we get $\gamma_{2rt}(G) = 4 = 2\gamma_{r2}(G)$. 

Now we show that the upper bound is strict, when $k \ge 3$,
i.e.~$\gamma_{krt}(G) < 2\gamma_{rk}(G)$ if $k \ge 3$.
Let $f$ be a $\rk$-function on $G$, i.e.~$\rk(G)=||f||=|V_1| + \cdots + |V_k| + \Sigma_{S \subseteq [k], |S| \ge 2}|S|\cdot|V_S|$. As pointed out above, from $f$ we can construct  a $k$RTDF $g$ by adding one label for each singleton, so the weight of $g$ is at most $2|V_1| + \cdots + 2|V_k| + \Sigma_{S \subseteq [k], |S| \ge 2}|S|\cdot|V_S|$. If there exists at least one vertex $v$ with $|f(v)|\geq 2$, we derive
$$2\rk(G)=2||f||>2|V_1| + \cdots + 2|V_k| + \Sigma_{S \subseteq [k], |S| \ge 2}|S|\cdot|V_S|\geq ||g||\geq \krt(G),$$
and we are done.
Thus now assume that $f$ assigns only singletons and empty sets to vertices of $G$.
We can assume that $f$ has at least one empty vertex, otherwise we could change all the labels to $\{ 1 \}$ to achieve a $k$RTDF of the same weight, so we would be done.

Let $f_1:V(G)\rightarrow 2^{[k]}$  be defined from $f$ as follows:
$$f_1(v) = 
\begin{cases}
\{ 1 \}, & \hbox{if $f(v) = \{ k -1 \}$} \\
\{ k-1, k \}, & \hbox{if $f(v) = \{ k \}$} \\
f(v), & \hbox{otherwise} \\
\end{cases}
$$

Now construct $f_2$ from $f_1$ according to the following:
\begin{itemize}
\item let $v$ be a vertex such that $f(v)=f_1(v)=\emptyset$, and define $f_2(v)=\{1\}$. Note that since $f$ is a $k$RDF there exist $x\in V_1$ and $y\in V_{k-1}$ such that $v$ is adjacent to both of them. Recall that $f_1(x)=f_1(y)=f_2(v)=\{1\}$,

\item for all non-empty vertices except $x, y,$ and $v$, add an element to it, so that it is no longer a singleton.

\end{itemize}
It is easy to see that $f_2$ is a $k$RTDF, thus we derive 
$$\krt(G)\leq ||f_2|| \leq 2|V_1| + 2|V_2| + \cdots + 2|V_k| -2 +1 = 2||f||-1<2\rk(G),$$
which concludes the proof.
\end{proof}

\noindent
The strict inequality in Proposition~\ref{rk-krt}  leads to the next question.

\begin{question}
Find a function $a(k)$ such that for every $k$ we have the tight bound
$$\gamma_{krt}(G) \le a(k) \cdot \gamma_{rk}(G).$$
\end{question}
From Proposition~\ref{rk-krt} we know $a(k) < 2$ for $k \ge 3$.
By taking $G$ to be $K_{k, b}$, for $b \ge k$, we get $\gamma_{rk}(G) = k$ and $\gamma_{krt}(G) > (3/2)k$,
by  Lemma~\ref{lem-Kab-2k},
and Proposition~\ref{thm3}, respectively; thus $a(k) > 3/2$.
In summary, we know that for $k \ge 3$, we must have 
$3/2 < a(k) < 2$.

\section{Lower bounding $k$-rainbow total domination}
\label{sec_dom}

The main point of this section is to lower bound $\krt(G)$ in terms of $\gamma(G)$.
In \cite{San19} it was observed that for a graph $G$ of order $n$ where $n > k > 1$, it is always the case that $\max\{k,\gamma(G)\}\leq \krt(G)$. 
While the lower bound of $k$ can be achieved, we will show that the lower bound of $\gamma(G)$ cannot.
 Goddard and Henning \cite{GodHen18} show that for any graph $G$, $\frac{4}{3}\gamma (G)$ is a tight lower bound for  $\trt(G)$.  We generalize their lower bound to all $k$ in the following theorem, where the interesting issue of tightness will be discussed after the theorem.

\begin{theorem} \label{thm_rainbow_vs_domination}
For a graph $G$ and $k \ge 2$ we have
$$\krt(G) \geq \dfrac{2k}{k+1} \gamma(G).$$
\end{theorem}

\begin{proof}

We use the notation $G[S]$ for the subgraph of $G$ induced by the set $S\subseteq V(G)$.
Let $f:V(G)\rightarrow 2^{[k]}$ be a $\krt$-function.
Let $V_i' \subseteq V_i$ be the set of vertices in $V_i$ having a neighbor in $V_i$, and let $D_i$ be a minimum dominating set for $G[V_i']$. Since $G[V_i']$ contains no isolated vertices, we have (due to a result of Ore in \cite{Ore}) that $| D_i | \le \frac{|V_i'|}{2} \leq \frac{|V_i|}{2}$. For $i\in [k]$ 
we define a set $U_i$ as follows:
$$U_i = V_i \ \cup \bigcup_{\substack{S \subseteq [k],\\ |S| \ge 2}} V_S \ \ \cup \ \ \bigcup_{j \neq i} D_j.$$
In other words, $U_i$ consists of the following set of vertices in $G$: those that are labeled by just $\{i\}$, those that are labeled by a subset of $[k]$ of size at least $2$,  and from among those labeled by $\{j\}$, where $j \neq i$, take just the ones in $D_j$. 

Now we show that each $U_i$ is a dominating set of $G$.
First, consider a vertex $v$ with $f(v)=\emptyset$. Since $f$ is a $k$RTDF, there is a neighbor $x$ of $v$
such that $i\in f(x)$. This means that $v$ is adjacent to a vertex in $V_i$ or a vertex labeled by a set of size $2$ or greater.  Therefore $v$ is adjacent to a vertex in $U_i$.  
Second, consider a vertex $u\in V_j\setminus D_j$ where $j \neq i$.  If $u$ is in $V_j'$ then it is adjacent to a vertex from $D_j$, otherwise, $u$ is not adjacent to a vertex labeled by $\{j\}$, so it must be adjacent to a vertex labeled by $S$ where $j \in S$ and $|S|\geq 2$.  Thus $u$ is adjacent to a vertex in $U_i$.
Finally, the rest of the vertices are actually in $U_i$.  Thus $U_i$ is a dominating set.
Now we can estimate as follows:

\begin{alignat*}{2} 
k \cdot \gamma(G) 
    & \le |U_1| + \ldots + |U_k| & \hspace{2cm} & 
		\\		
    & = |V_1| + \cdots + |V_k| + k \cdot  \sum_{\substack{S \subseteq [k],\\ |S| \ge 2}} |V_S| +
(k-1)(|D_1| + \ldots + |D_k|)\\ 
    & \le |V_1| + \cdots + |V_k| +  \sum_{\substack{S \subseteq [k],\\ |S| \ge 2}} k \cdot|V_S|+
\frac{k-1}{2}(|V_1| + \ldots + |V_k|) & \hspace{2cm} & 
\\ 
    & = \sum_{S \subseteq [k]} |S| \cdot |V_S|   
+ \sum_{\substack{S \subseteq [k],\\ |S| \ge 2}} (k - |S|) |V_S| +
\frac{k-1}{2}(|V_1| + \ldots + |V_k|)\\ 
    & \le \gamma_{krt}(G)  + 
(k-1)\sum_{\substack{S \subseteq [k],\\ |S| \ge 2}} |V_S| +
\frac{k-1}{2}(|V_1| + \ldots + |V_k|)& \hspace{2cm} & 
\\
    & = \gamma_{krt}(G)  
		+\frac{k-1}{2}\left( \ \left(\sum_{\substack{S \subseteq [k],\\ |S| \ge 2}} 2 \cdot |V_S| \right)+  \ |V_1| + \ldots + |V_k| \ \right)\\
		& \le \gamma_{krt}(G) + \frac{k-1}{2} \gamma_{krt}(G) & \hspace{2cm} & 
		\\ 		
		& = \frac{k+1}{2} \gamma_{krt}(G).  			
\end{alignat*}
Summarizing, we have 
$k \gamma(G) \le \frac{k+1}{2} \gamma_{krt}(G)$,
thus $\frac{2k}{k+1}\gamma(G) \le \gamma_{krt}(G)$. 
\end{proof}

\noindent
In what follows we consider the tightness of the bound in the above theorem. The tightness for $k=2$ follows from the mentioned result by Goddard and Henning; see Theorem 3.1 in~\cite{GodHen18}. To see that the lower bound is tight for $k=3$, we modify a construction from \cite{GodHen18} in order to define
a family of graphs $\Sgraph_{m}$, where $m \ge 2$ is an integer parameter.
A graph $G$ is in the family $\Sgraph_{m}$ if we can construct it  as follows.  There are two parts to $G$: $H$ and $I$.
$H$ consists of the three graphs $H_1, H_2$, and $H_3$, where each $H_i$ is $m$ disjoint edges.
To construct $I$, for each $3$-tuple of vertices $\bar{x} = (x_1, x_2, x_3$), where $x_i \in H_i$,
add an independent set of size at least $m$ to $I$, calling it $I_{\bar{x}}$. Attach each $x_i$ to all vertices in $I_{\bar{x}}$. 
See Figure~\ref{t3} for an illustration of a graph from $\Sgraph_{3}$; for clarity, only $3$ of the $6^3$ independent sets are depicted.
We calculate $\gamma_{3rt}(G)$ for $G$ in $\Sgraph_{m}$, and then state the immediate and interesting corollary.

\begin{figure}[ht!]
\begin{center}
\begin{tikzpicture}[scale=0.8]

\draw[black,rounded corners=10]
     (-2,-2) rectangle (13,4);
		\draw (13.3,1) node {$H$};

 \draw[black,rounded corners=10]
     (-1,-1) rectangle (2,3);
		\draw (0.5,3.3) node {$H_1$};
		
 \draw[black,rounded corners=10]
     (4,-1) rectangle (7,3);
		\draw (5.5,3.3) node {$H_2$};
		
	\draw[black,rounded corners=10]
     (9,-1) rectangle (12,3);
		\draw (10.5,3.3) node {$H_3$};

\node [My Style, name=a, label=left:$x_1^1$ ]   at (0,2) {};		
\node [My Style, name=b, label=right:$x_2^1$]   at (1,2) {};	
\node [My Style, name=c, label=left:$x_3^1$]    at (0,1) {};	
\node [My Style, name=d, label=right:$x_4^1$]   at (1,1) {};	
\node [My Style, name=e, label=left:$x_5^1$]    at (0,0) {};	
\node [My Style, name=f, label=right:$x_6^1$]   at (1,0) {};	
		
\draw[thick] (a) -- (b);
\draw[thick] (c) -- (d);
\draw[thick] (e) -- (f);

\node [My Style, name=a1, label=left:$x_1^2$ ]   at (5,2) {};		
\node [My Style, name=b1, label=right:$x_2^2$]   at (6,2) {};	
\node [My Style, name=c1, label=left:$x_3^2$]    at (5,1) {};	
\node [My Style, name=d1, label=right:$x_4^2$]   at (6,1) {};	
\node [My Style, name=e1, label=left:$x_5^2$]    at (5,0) {};	
\node [My Style, name=f1, label=right:$x_6^2$]   at (6,0) {};	
		
\draw[thick] (a1) -- (b1);
\draw[thick] (c1) -- (d1);
\draw[thick] (e1) -- (f1);

\node [My Style, name=a2, label=left:$x_1^3$ ]   at (10,2) {};		
\node [My Style, name=b2, label=right:$x_2^3$]   at (11,2) {};	
\node [My Style, name=c2, label=left:$x_3^3$]    at (10,1) {};	
\node [My Style, name=d2, label=right:$x_4^3$]   at (11,1) {};	
\node [My Style, name=e2, label=left:$x_5^3$]    at (10,0) {};	
\node [My Style, name=f2, label=right:$x_6^3$]   at (11,0) {};
	
\draw[thick] (a2) -- (b2);
\draw[thick] (c2) -- (d2);
\draw[thick] (e2) -- (f2);	

\draw[black,rounded corners=10]
     (-3,-8.6) rectangle (14,-3);
		\draw (14.3,-6) node {$I$};
		
  \draw[black,rounded corners=10]
     (-2,-7.5) rectangle (-1,-3.5);
		\draw (-1.5,-8) node {$I_{(x_1^1,x_1^2,x_1^3)}$};

   \draw (1,-6) node {$\cdots$};
	
	\draw[black,rounded corners=10]
     (3,-7.5) rectangle (4,-3.5);
		\draw (3.5,-8) node {$I_{(x_1^1,x_4^2,x_5^3)}$};

 \draw (8,-6) node {$\cdots$};

\draw[black,rounded corners=10]
     (12,-7.5) rectangle (13,-3.5);
		\draw (12.5,-8) node {$I_{(x_6^1,x_6^2,x_6^3)}$};

\node [My Style, name=x]   at (-1.5,-4) {};		
\node [My Style, name=y]   at (-1.5,-5) {};	
\node [My Style, name=u]   at (-1.5,-6) {};	
\node [My Style, name=v]   at (-1.5,-7) {};

\draw[very thin] (x) -- (a);
\draw[very thin] (y) -- (a);
\draw[very thin] (u) -- (a);
\draw[very thin] (v) -- (a);

\draw[very thin] (x) -- (a1);
\draw[very thin] (y) -- (a1);
\draw[very thin] (u) -- (a1);
\draw[very thin] (v) -- (a1);

\draw[very thin] (x) -- (a2);
\draw[very thin] (y) -- (a2);
\draw[very thin] (u) -- (a2);
\draw[very thin] (v) -- (a2);

\node [My Style, name=x1]   at (3.5,-4) {};		
\node [My Style, name=y1]   at (3.5,-5) {};	
\node [My Style, name=u1]   at (3.5,-6) {};	

\draw[very thin] (x1) -- (a);
\draw[very thin] (y1) -- (a);
\draw[very thin] (u1) -- (a);

\draw[very thin] (x1) -- (d1);
\draw[very thin] (y1) -- (d1);
\draw[very thin] (u1) -- (d1);

\draw[very thin] (x1) -- (e2);
\draw[very thin] (y1) -- (e2);
\draw[very thin] (u1) -- (e2);

\node [My Style, name=x5]   at (12.5,-4) {};		
\node [My Style, name=y5]   at (12.5,-5) {};	
\node [My Style, name=u5]   at (12.5,-6) {};	
\node [My Style, name=v5]   at (12.5,-7) {};

\draw[very thin] (x5) -- (f);
\draw[very thin] (y5) -- (f);
\draw[very thin] (u5) -- (f);
\draw[very thin] (v5) -- (f);

\draw[very thin] (x5) -- (f1);
\draw[very thin] (y5) -- (f1);
\draw[very thin] (u5) -- (f1);
\draw[very thin] (v5) -- (f1);

\draw[very thin] (x5) -- (f2);
\draw[very thin] (y5) -- (f2);
\draw[very thin] (u5) -- (f2);
\draw[very thin] (v5) -- (f2);		
		
\end{tikzpicture}
\end{center}
\caption{A graph from $\Sgraph_{3}$.}
\label{t3}
\end{figure}

\begin{proposition} \label{thm_special_graph}
Suppose $G$ is in the family
$\Sgraph_{m}$.  Then
$$  \gamma_{3rt}(G) = \dfrac{3}{2}\gamma(G).$$
\end{proposition}

\begin{proof} 

We obtain $\gamma_{3rt}(G) \ge \dfrac{3}{2}\gamma(G)$ from Theorem~\ref{thm_rainbow_vs_domination}.  The rest of the proof shows that $\gamma_{3rt}(G) \le \dfrac{3}{2}\gamma(G)$. 
It suffices to show that $\gamma_{3rt}(G)  \le 6m$ and $\gamma(G) \ge 4m$.  For the first inequality, let
$f: V(G)\rightarrow 2^{[3]}$ be defined by $f(v)=\{i\}$ if $v\in V(H_i)$ and $f(v)=\emptyset$ for every vertex $v$ in $I$. It is easy to see that $f$ is a $3$RTDF of $G$ thus $\gamma_{3rt}(G)\leq ||f|| \le 6m$.  The rest of the proof shows that
$\gamma(G) \ge 4m$.

Let $D$ be a minimum size dominating set of $G$, where $D$ is chosen to have the following property: $D$ contains 
as many vertices from 
$H_1 \cup H_2 \cup H_3$ as possible.  We claim that such a $D$ must contain every vertex of $H_i$ for some $i =1, 2, $ or $3$. 
Assume for contradiction that there is no such $H_i$; then there must be a tuple $(u_1, u_2, u_3)$ such that $u_i \in H_i$ and no $u_i$ is in $D$.  The structure of $G$ implies that every vertex of $I_{(u_1, u_2, u_3)}$ must be in $D$.  However the closed neighborhood of $I_{(u_1, u_2, u_3)}$ contains the vertices $I_{(u_1, u_2, u_3)}$ along with $\{ u_1, u_2, u_3\}$, while the closed neighborhood of $\{  u_1, u_2, u_3  \}$ contains the same vertices as well as others.  Since $I_{(u_1, u_2, u_3)}$ contains at least $3$ vertices, we could remove $I_{(u_1, u_2, u_3)}$ from $D$ and replace them by $\{ u_1, u_2, u_3\}$ to arrive at a different minimum dominating set that contains more vertices of $H_1 \cup H_2 \cup H_3$, contradicting the property of $D$ having as many vertices from $H_1 \cup H_2 \cup H_3$ as possible.

Thus we assume $D$ contains all of $H_1$.
The vertices of $H_1$ are not adjacent to the vertices of $H_2 \cup H_3$, so we still must dominate all of them.
Note that for any vertex in the graph, it dominates at most two vertices in
$H_2 \cup H_3$, thus to dominate the $4m$ vertices of $H_2 \cup H_3$ we need at least $2m$ more vertices in addition to the $2m$ vertices of $H_1$.  Thus $\gamma(G) \ge 4m$.

\end{proof}

\begin{cor} \label{cor_3rt_bound}
For a graph $G$ we have the following tight inequality:
$$\gamma_{3rt}(G) \ge \dfrac{3}{2}\gamma(G).$$
\end{cor}

The construction of $\Sgraph_{m}$, along with the last corollary does not obviously generalize to $k \ge 4$.
Summarizing, we have the following tight lower bounds for the cases of $k = 2$ and $k = 3$: $\gamma_{2rt}(G) \ge \frac{4}{3}\gamma(G)$ and $\gamma_{3rt}(G) \ge \frac{3}{2}\gamma(G)$.  We make the following conjecture for the other values of $k$.

\begin{conjecture} \label{conj_2gammaLower}
For a graph $G$ and $k \ge 4$, we have the tight bound $\krt(G)\geq 2\gamma(G)$.
\end{conjecture}

\noindent
In Conjecture~\ref{conj_2gammaLower} the correct coefficient $c$ in front of $\gamma(G)$ (which we conjecture to be 2) satisfies 
$\frac{2k}{k+1} \le c \le 2$.  The first inequality holds by Theorem~\ref{thm_rainbow_vs_domination}.
The second inequality holds because $\gamma_{\rm{k}rt}(K^+_{k, b}) = 2k = 2\cdot \gamma(K^+_{k, b})$, by
Lemma~\ref{thm_Kplus} (for $k \ge 2, b \ge 1$).

In the case of bipartite graphs, when $k=2$, the bound in Theorem~\ref{thm_rainbow_vs_domination} can be improved. Using hypergraphs Azarija et al.~showed that for a bipartite graph $G$, $\gamma_{2rt}(G) = 2\gamma(G)$ (see Theorem 1 in~\cite{azar}). Goddard and Henning~\cite{GodHen18} (see Theorem 2.1) presented a shorter proof using paired domination. Using $2$-rainbow total domination the proof of the mentioned result is even simpler.

\begin{theorem}[\cite{azar,GodHen18}] \label{azar}
If $G$ is a bipartite graph then
         $\gamma_{2rt}(G) = 2\gamma(G)$.

\end{theorem}

\begin{proof}
Let $G$ be a bipartite graph with bipartition $(L,R)$, and let $f$ be a $\trt$-function for $G$. We define a set $A$ as the union of the following three subsets of $V(G)$: all vertices of $G$ with label $\{1,2\}$, all vertices from $L$ with label $\{2\}$, and   all vertices from $R$ with label $\{1\}$.
Obviously, $A$ is a dominating set in $G$. By interchanging $L$ and $R$ in the above definition, a set $B$ is defined, which is also a dominating set in $G$. Therefore 
$$2\gamma(G) \leq |A|+|B|=2 |V_{12}|+|V_1|+|V_2|=||f||.$$ 
The opposite inequality $||f|| \le 2 \gamma(G)$ is immediate (see Observation 2.3 in \cite{San19}).
\end{proof}




\section{Vizing-like conjectures}
\label{sec_vizing}

Vizing's well known conjecture states that for any graphs $G$ and $H$ 
        \begin{equation}\label{in.v}
				\gamma(G) \, \cdot \, \gamma(H) \le  \gamma(G\cp H).
				\end{equation}
By a result of Clark and Suen \cite{CS2000}	it is known that for any graphs $G$ and $H$ 
$$\gamma(G) \, \cdot \, \gamma(H) \le 2 \gamma(G\cp H).$$
Ho~\cite{Ho08} proved that  for any graphs $G$ and $H$ without isolated vertices, 
        $$
				\gamma_t(G) \, \cdot \, \gamma_t(H) \le  2 \cdot \gamma_t(G\cp H),
				$$
where the inequality is tight.  We use above results to provide a simple proof of the following Vizing-like inequalities for $k$-rainbow and $k$-rainbow total domination.

\begin{proposition} 
Let $G$ and $H$ be graphs and $k\geq 2$. Then 
$$
     \krt(G) \, \cdot \, \krt(H) \le 2k \cdot \krt(G\cp H)
$$
and 
$$
     \rk(G) \, \cdot \, \rk(H) \le 2k \cdot \rk(G\cp H). 
$$\end{proposition} 

\begin{proof}
Using Ho's result, the equality in equation~(\ref{cpt}),  and Corollary~\ref{thm_rainbow_to_total} we derive 
\begin{alignat*}{2} 
\krt(G) \, \cdot \, \krt(H) 
    & = \gamma_t(G \cp K_k) \, \cdot \, \gamma_t(H \cp K_k) && \text{} \\		
    & \le 2\gamma_t(G \cp K_k \cp H \cp K_k)                && \text{} \\	
		& = 2\gamma_t(G \cp H \cp K_k \cp K_k)                  && \text{} \\ 
	  & = 2\krt(G \cp H \cp K_k)                              && \text{} \\ 		
	  & \le 2k\gamma_t(G \cp H \cp K_k)                       && \text{} \\	 
	  & = 2k\krt(G \cp H).                                 && \text{} 
\end{alignat*}
Following the same lines, using the above result of Clark and Suen, and known bound $\rk(G)\leq k\gamma(G)$ from \cite{hart04}, the Vizing-like inequality for $k$-rainbow domination can be derived.
\end{proof}

\noindent
We believe that the following stronger inequalities might hold. 

\begin{conjecture}\label{conB}
Let $G$ and $H$ be graphs and $k\geq 2$. Then 
$$\rk(G) \, \cdot \, \rk(H) \le  2 \cdot \rk(G\cp H).$$
\end{conjecture}

\begin{conjecture}\label{conA}
Let $G$ and $H$ be graphs and $k\geq 2$. Then 
 \begin{equation}\label{in.a}
     \krt(G) \, \cdot \, \krt(H) \le  2 \cdot \krt(G\cp H).
 \end{equation} 
\end{conjecture}

For the first conjecture
the constant 2 is attained for $G=H=C_4$ and $k=4$. It is easy to see that $\gamma_{{\rm r}4}(C_4)=4$.
To see that $\gamma_{{\rm r}4}(C_4\cp C_4) = 8$, observe that the upper bound
 $\gamma_{{\rm r}4}(C_4\cp C_4)\leq 8$ follows from the construction in Figure \ref{c4c4}, while the lower bound 
$\gamma_{{\rm r}4}(C_4\cp C_4)\geq 8$ follows from Theorem~3 in \cite{shao14} 
(that theorem states that for a connected graph $G$, $\rk(G)\geq \Bigl\lceil \frac{|V(G)|k}{\Delta(G)+k}\Bigr\rceil$, where $\Delta(G)$ stands for the maximum degree of $G$).

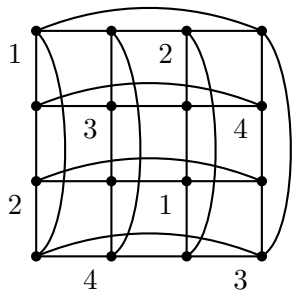
\begin{figure}[ht!]
	\begin{center}
		\begin{tikzpicture}[scale=1,style=thick]

		\node [My Style, name=a]   at (0,0) {};
		\node [My Style, name=b, label=below left:$4$]   at (1,0) {};		
		\node [My Style, name=c]   at (2,0) {};	
		\node [My Style, name=d, label=below left:$3$]   at (3,0) {};
		
		\node [My Style, name=a1, label=below left:$2$]   at (0,1) {};
		\node [My Style, name=b1]   at (1,1) {};		
		\node [My Style, name=c1, label=below left:$1$]   at (2,1) {};	
		\node [My Style, name=d1]   at (3,1) {};
		
		\node [My Style, name=a2]   at (0,2) {};
		\node [My Style, name=b2, label=below left:$3$]   at (1,2) {};		
		\node [My Style, name=c2]   at (2,2) {};	
		\node [My Style, name=d2, label=below left:$4$]   at (3,2) {};
		
		\node [My Style, name=a3, label=below left:$1$]   at (0,3) {};
		\node [My Style, name=b3]   at (1,3) {};		
		\node [My Style, name=c3, label=below left:$2$]   at (2,3) {};	
		\node [My Style, name=d3]   at (3,3) {};
		

		\draw[] (a)--(b)--(c)--(d);		
		\draw[] (a1)--(b1)--(c1)--(d1);
		\draw[] (a2)--(b2)--(c2)--(d2);
		\draw[] (a3)--(b3)--(c3)--(d3);
		
		\draw (a) .. controls (1,0.4) and (2,0.4) .. (d);
		\draw (a1) .. controls (1,1.4) and (2,1.4) .. (d1);		
		\draw (a2) .. controls (1,2.4) and (2,2.4) .. (d2);
		\draw (a3) .. controls (1,3.4) and (2,3.4) .. (d3);
		
		\draw[] (a)--(a1)--(a2)--(a3);		
		\draw[] (b)--(b1)--(b2)--(b3);
		\draw[] (c)--(c1)--(c2)--(c3);
		\draw[] (d)--(d1)--(d2)--(d3);
		
		\draw (a) .. controls (0.5,0.4) and (0.5,2.4) .. (a3);
		\draw (b) .. controls (1.5,0.4) and (1.5,2.4) .. (b3);
		\draw (c) .. controls (2.5,0.4) and (2.5,2.4) .. (c3);
		\draw (d) .. controls (3.5,0.4) and (3.5,2.4) .. (d3);
		
	\end{tikzpicture}
	\end{center}
	\caption{$4$-rainbow domination of $C_4 \cp C_4$.}
	\label{c4c4}
\end{figure}

Regarding Conjecture~\ref{conA}, if we exclude graphs containing isolated vertices and take $k=1$, then inequality~(\ref{in.a}) is just a restatement of Ho's result.  Furthermore, when $k=1$, Ho observes that the inequality is sharp.
When $k=2$ we obtain an interesting observation relating Vizing's conjecture to Conjecture~\ref{conA}.
Due to Theorem~\ref{azar}, if  $G$ and $H$ are bipartite graphs, then 
 $$4 \gamma(G) \gamma(H) = \trt(G) \trt(H) \quad\hbox{ and }\quad
   4\gamma(G\cp H) = 2 \trt (G\cp H),$$
which implies that for bipartite graphs, inequality (\ref{in.v}) holds if and only if for $k = 2$ inequality (\ref{in.a}) holds.
Furthermore, one inequality is tight if and only if the other is. 
The next observation summarizes this discussion.
   
\begin{observation}
Restricting to the case $k=2$ and to the class of bipartite graphs,  Conjecture~\ref{conA} and Vizing's conjecture are equivalent.
\end{observation}

\noindent
The relevance of the last observation is highlighted by the fact that there are
many pairs of graphs for which equality is achieved in Vizing's conjecture (see \cite{hart98}).


\vskip 1pc \noindent{\bf Acknowledgments.} 
Slovenian researchers were partially supported by Slovenian research agency
ARRS, program no.\ P1--0383, project no.\ J1--1692 and bilateral projects between Slovenia and United states of America BI-US/18-20-061 and BI-US/18-20-052.


\begin{thebibliography}{99}






\bibitem{azar}
J.~Azarija, M.~A.~Henning, S. Klav\v{z}ar, (Total) Domination in Prisms, {\it Electron. J. Comb.} {\bf 24(1)} (2017) \#P1.19.


\bibitem{viz}
B.~Bre\v{s}ar, P.~Dorbec, W.~Goddard, B.~L.~Hartnell, M.~A.~Henning, S.~ Klav\v{z}ar, D.~F.~Rall, Vizing's conjecture: a survey and recent results. {\it J. Graph Theory} {\bf 69} (2012) 46--76.

\bibitem{btks-2007}
B.~Bre\v sar, T.~Kraner \v{S}umenjak,
On the $2$-rainbow domination in graphs,
{\it Discrete Appl. Math.} {\bf 155}  (2007) 2394--2400.

\bibitem{bhr-2008}
B.~Bre\v sar, M.~A.~Henning, D. F. Rall, Rainbow domination in
graphs, {\it Taiwanese J. Math.} {\bf 12} (2008) 213--225.

\bibitem{chang10}
G.~J.~Chang, J.~Wu, X.~Zhu, Rainbow domination on trees, 
{\it Discrete Appl. Math.} {\bf 158}(2010) 8--12.

\bibitem{chang13}
G.~J.~Chang, B.~Li, J.~Wu,
Rainbow domination and related problems on strongly chordal graphs,
{\it Discrete Appl. Math} {\bf 161} (2013) 1395--1401.


\bibitem{CS2000}
W.~E.~Clark and S.~Suen, An inequality related to Vizing’s conjecture,
{\it Electron. J. Combin } {\bf7(Note 4)} (2000), 3pp.

\bibitem{fund-1998} 
T.~W.~Haynes, S.~T.~Hedetniemi, and P.~J.~Slater. \emph{Fundamentals of Domination in Graphs},
Marcel Dekker, New York, 1998.

\bibitem{HenSur09}
M.~A.~Henning, A survey of selected recent results on total domination in graphs, {\it Discrete Math.} {\bf 309} (2009) 32--63.

\bibitem{GodHen18}
W.~Goddard, M.~A.~Henning,
A note on domination and total domination in prisms,
{\it J. Comb. Optim.} {\bf 35} (2018) 14--20.

\bibitem{hart98}
B.~Hartnell and D.~F.~Rall,
Domination in Cartesian products: Vizing’s
conjecture, Domination in graphs, {\it volume 209 of Monogr Textbooks Pure
Appl Math}, Dekker, New York (1998), 163--189.

\bibitem{hart04}
B.~Hartnell, D.~F.~Rall, 
On dominating the Cartesian product of a graph and $K_2$,
{\it Discuss. Math. Graph Theory} {\bf 24} (2004) 389--402.

\bibitem{Ho08}
P.~T.~Ho, A note on the total domination number, {\it Utilitas Mathematica } {\bf 77} (2008) 97--100.

\bibitem{LuHou}
Y.~Lu and X.~Hou, Total domination in the Cartesian product of a graph and $K_2$ or $C_n$, {\it Utilitas Math.} {\bf 83} (2010) 313--322.

\bibitem{Ore}
O.~Ore, {\it Theory of graphs}, Amer. Math. Spc. Colloq. Publ., 3 (1962).

\bibitem{philip}
M. Pilipczuk, M. Pilipczuk and R.~\v{S}krekovski, 
Some results on Vizing’s conjecture and related problems, 
{\it Discrete Appl. Math} {\bf 160} (2012) 2484--2490.


\bibitem{shao14}
Z.~Shao, M.~Liang, C.~Yin, X.~Xu, P.~Pavli\v{c}, J.~\v{Z}erovnik
On rainbow domination numbers of graphs,  
{\it Information Sciences} {\bf 254} (2014) 225--234.

\bibitem{filo}
Z.~Shao, S.~M.~Sheikholeslami, B.~Wang, P.~Wu, X.~Zhang,
Trees with Equal Total Domination and 2-Rainbow Domination Numbers,  
{\it Filomat} {\bf 32:2} (2018) 599--607.




\bibitem{TKSRT-lex}
T.~K.~\v{S}umenjak, D.~F.~Rall, A.~Tepeh,
Rainbow domination in the lexicographic product of graphs, 
{\it Discrete Appl. Math} {\bf 161} (2013) 2133--2141.

\bibitem{San19}
A.~Tepeh, Total domination in generalized prisms and a new domination
invariant, {\it to appear in Discuss. Math. Graph Theory}.

\bibitem{greedy}
D.~Zhao, Z.~Shu,
A greedy algorithm for $k$-Rainbow domination numbers of graphs,
{\it Journal of Computational and Theoretical Nanoscience} {\bf 13(5)} 
(2016) 3374--3377.

\bibitem{wx-2010} 
Y.~Wu and H.~Xing,
Note on 2-rainbow domination and Roman domination in graphs,
{\it Appl. Math. Lett.} {\bf 23} (2010) 706--709.



\end{thebibliography}
\end{document}